\documentclass[reqno,11pt]{amsart}
\usepackage[left=1in,right=1in, bottom=1in, top=1in]{geometry}
\usepackage{amssymb, amsthm}
\usepackage{amsmath}
\usepackage[numbers]{natbib}
\usepackage{mathtools}
\usepackage[dvipsnames]{xcolor}
\usepackage{enumerate}

\input epsf
\theoremstyle{plain}
\newtheorem{thm}{Theorem}[section]
\newtheorem{lem}[thm]{Lemma}

\newtheorem{prop}[thm]{Proposition}

\theoremstyle{definition}

\newtheorem{defn}[thm]{Definition}
\newtheorem{quest}[thm]{Question}

\newtheorem{rem}[thm]{Remark}

\usepackage{setspace}
\DeclareMathOperator{\ssm}{\smallsetminus}

\DeclareMathOperator{\A}{\mathcal{A}}
\DeclareMathOperator{\lk}{lk}
\DeclareMathOperator{\al}{\alpha}
\DeclareMathOperator{\be}{\beta}

\DeclareMathOperator{\id}{Id}

\DeclareMathOperator{\N}{\mathbb{N}}
\DeclareMathOperator{\Z}{\mathbb{Z}}

\DeclareMathOperator{\Q}{\mathbb{Q}}

\DeclareMathOperator{\spa}{span}

\DeclareMathOperator{\Bl}{Bl}

\usepackage{xspace}
\usepackage{hyperref}

\makeatletter
\DeclareRobustCommand\onedot{\futurelet\@let@token\@onedot}
\def\@onedot{\ifx\@let@token.\else.\null\fi\xspace}
\def\eg{{e.g}\onedot} 
\def\ie{{i.e}\onedot}

\makeatother

\begin{document}

\title{Branched covers bounding rational homology balls}

\overfullrule5mm
\author[P.\ Aceto]{Paolo Aceto}
\address{Mathematical Institute, University of Oxford, Oxford, United Kingdom}
\email{paoloaceto@gmail.com}

\author[J.\ Meier]{Jeffrey Meier}
\address{Department of Mathematics,  Western Washington University, Bellingham, WA, United States}
\email{jeffrey.meier@wwu.edu}

\author[A.\ N.\ Miller]{Allison N. Miller}
\address{Department of Mathematics, Rice University, Houston, TX, United States}
\email{allison.miller@rice.edu}

\author[M.\ Miller]{Maggie Miller}
\address{Department of Mathematics, Princeton University, Princeton, NJ, United States}
\email{maggiem@math.princeton.edu}

\author[J.\ Park]{JungHwan Park}
\address{School of Mathematics, Georgia Institute of Technology, Atlanta, GA, United States}
\email{junghwan.park@math.gatech.edu }

\author[A.\ I.\ Stipsicz]{Andr\'as I. Stipsicz}
\address{R\'{e}nyi Institute of Mathematics, Budapest, Hungary}
\email{stipsicz.andras@renyi.hu}

\begin{abstract}
Prime power fold cyclic branched covers along smoothly slice knots
all bound rational homology balls. This phenomenon, however, does not
characterize slice knots. In this paper, we give {a new construction} of non-slice
knots that have the above property. The sliceness obstruction comes from computing twisted Alexander polynomials, and we introduce new techniques to simplify their calculation.
\end{abstract}
\maketitle

\section{Introduction}
\label{sec:intro}

For a knot $K\subset S^3$, let $\Sigma _q (K)$ denote the $q$-fold
cyclic branched cover of $S^3$ along~$K$. Consider the set of prime
powers ${\mathcal {Q}}=\{ p^\ell \mid p {\mbox { prime, }} \ell \in \N
\}$.  For $q\in\mathcal{Q}$, the three-manifold~$\Sigma _q(K)$ is a
rational homology sphere -- \ie $H_*(\Sigma _q (K); \Q)\cong H_* (S^3
; \Q)$. It is not hard to see that if $K\subset S^3$ is smoothly slice
-- \ie bounds a smooth, properly embedded disk~$D$ in the 4-ball $D^4$
-- then~$\Sigma _q(K)$ bounds a smooth rational homology ball $X^4$,
that is, $\Sigma _q(K)=\partial X^4$ and $H_*(X^4; \Q)\cong H_* (D^4 ;
\Q)$. Indeed, the $q$-fold cyclic branched cover of $D^4$ branched
along $D$ will be such a four-manifold. It is natural to ask if the
property that all prime power fold cyclic branched covers bound
rational homology balls characterizes slice knots (see \eg
\cite{banff-questions,bonn-questions}).

To put this question in a more algebraic framework, notice that
$\Sigma _q \left(-K\right)=-\Sigma _q(K)$ (where $-K$ is the reverse of the mirror image of the knot $K$ and $-Y$ is the
three-manifold $Y$ with reversed orientation) and $\Sigma _q (K_1\#
K_2)=\Sigma _q (K_1)\# \Sigma _q (K_2)$. Hence the map
\[
K\mapsto \Sigma _q (K)
\]
descends to a homomorphism ${\mathcal {C}}\to \Theta ^3 _{\Q}$, where 
${\mathcal {C}}$ denotes the smooth concordance group of knots in~$S^3$, and 
$\Theta ^3 _{\Q}$ is the smooth rational homology cobordism group
of rational homology spheres. We then let   
$$
\varphi \colon   {\mathcal {C}}\to \prod _{q\in {\mathcal {Q}}}\Theta ^3_{\Q},
$$
be the homomorphism given by 
\[ [K]\mapsto   ([\Sigma _q(K)])_{q\in {\mathcal {Q}}}, \]
and note that  $[K] \in \ker \varphi$ exactly when all the  prime power fold cyclic branched covers of $K$ bound rational homology balls. In this article, we give a new construction that yields large families of knots representing elements in $\ker\varphi$.
 
If $K$ is a knot that is \emph{not} concordant to its reverse $K^r$, then $K\#-K^r$ is non-slice and represents a non-trivial element in $\ker\varphi$, since $\Sigma_q(K\#-K^r)\cong\Sigma_q(K)\#-\Sigma_q(K)$ always bounds a rational homology ball when $q\in\mathcal Q$.
 The existence of such knots was first shown by Livingston; see~\cite{Kirk-Livingston:1999-2,Kirk-Livingsont:2001-1} for proofs.
 In particular, recent work of Kim and Livingston implies that $\ker\varphi$ contains an infinite free subgroup generated by topologically slice knots of the form $K\#-K^r$~\cite{Kim-Livingston:2019-1}.
 
Considerably less seems to be known with regards to finite order elements in $\ker\varphi$. Kirk and Livingston showed that the knot $8_{17}$, which is negative-amphichiral, is not concordant to its reverse; hence $8_{17}\#8_{17}^r$ represents a nontrivial element of order two in $\ker\varphi$~\cite{Kirk-Livingston:1999-2}; see also,~\cite{ColKirLiv_15}. In the present article, we extend this result by showing that there exists a subgroup $H$ of $\ker\varphi$ such that $H$ is isomorphic to $(\Z_2)^5$; see Theorem~\ref{thm:main2} below.
%


Our examples are constructed as follows. Let $L_r$ be the link depicted in the left diagram of Figure~\ref{fig:knot}, where the box
labeled $r\in \N $ consists of $r$ right-handed half-twists (and $-r$
denotes $r$ left-handed half-twists).  When $r$ is even, $L_r$ is a
knot (a simple generalization of the figure-8 knot, which is given by
$L_2$). As was shown in~\cite{Cha}, these knots are rationally slice,
non-slice, and strongly negative-amphichiral and moreover generate a
subgroup isomorphic to $(\Z _2)^{\infty}$ in the smooth concordance
group ${\mathcal {C}}$.  If $r=2m+1$ is odd, then $L_r$ is a
2-component link of unknots, which we redraw in the middle of
Figure~\ref{fig:knot} by braiding component $B_{2m+1}$ about
component~$A_{2m+1}$. The resulting $(2m+1)$-braid $\beta _m$ is shown
in the right diagram of Figure~\ref{fig:knot}.

\begin{figure}[h]
\includegraphics[width=12cm]{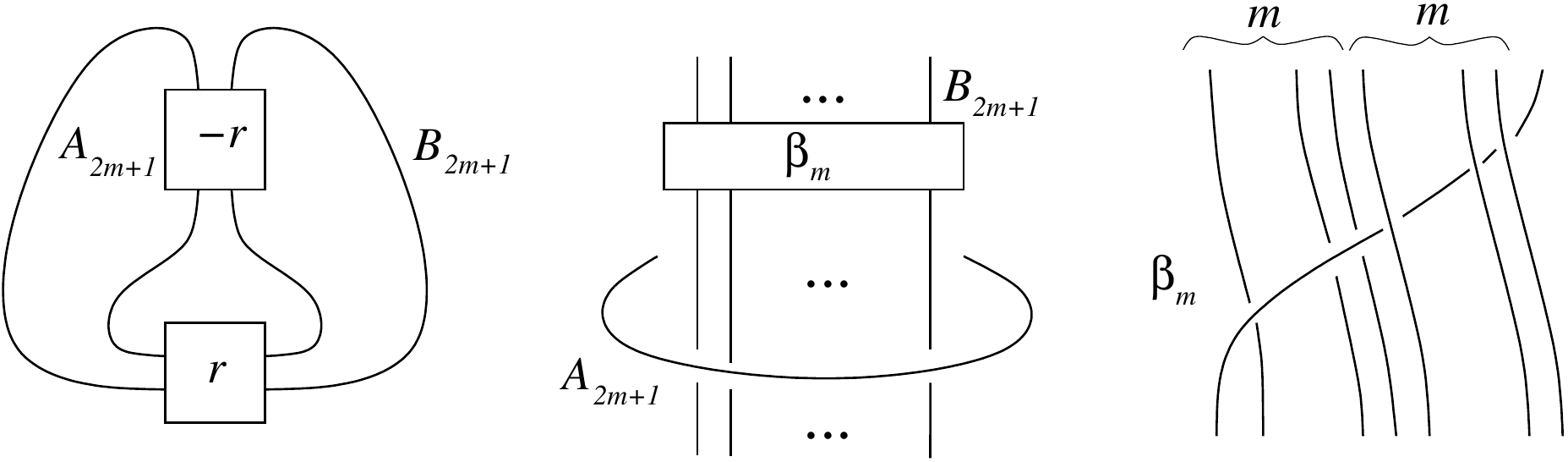}
\caption{$L_r$ (left) is a knot if $r$ is even and is a 2-component link if $r=2m+1$ is odd.
The middle diagram shows $L_{2m+1}=A_{2m+1} \cup B_{2m+1}$ redrawn as (the closure of) a $(2m+1)$-braid with its braid axis. On the right we give the $(2m+1)$-braid $\beta _m$.}
\label{fig:knot}
\end{figure}

We define $K_{m,n}$ to be the lift of $B_{2m+1}$ to
$\Sigma_n(A_{2m+1})$, which since $A_{2m+1}$ is an unknot is just
$S^3$.  Note that $K_{m,n}$ is a knot if $r=2m+1$ and $n$ are
relatively prime. In fact, the description of Figure~\ref{fig:knot}
shows that $K_{m,n}$ is simply the braid closure of the braid $\beta
_m ^n$. {We use the symmetry of $L_{2m+1}$ to show that
  $\Sigma_q(K_{m,n})$ is diffeomorphic to $\Sigma_n(K_{m,q})$ when $n$
  and $q$ are both relatively prime to $2m+1$. We then use the fact
  that $K_{m,n}$ is strongly negative-amphichiral to show that many of
  these knots represent elements of $\ker \varphi$.}

\begin{thm}\label{thm:main1}
If $n$ is an odd prime power which is relatively prime to $2m+1$, then
$[K_{m,n}]\in\ker \varphi.$
\end{thm}

For instance, if $n$ is an odd prime power and not divisible by $3$,
then $K_{1,n}$ is contained in $\ker \varphi$. The knots $K_{1,n}$
previously appeared in work of Lisca~\cite{lisca}, where it was
pointed out that these knots are strongly
negative-amphichiral. Therefore they are of order at most two in
$\mathcal{C}$. In addition, Sartori proved in his
thesis~\cite{Sartori} that one of these knots ($K_{1,7}$ in our
notation) is not slice; hence, by Theorem~\ref{thm:main1}, this knot spans $\Z_2 \leq \ker \varphi$. We extend Sartori's non-sliceness result to show that some other members of the family
represent non-trivial elements in $\ker \varphi$; moreover, we show
that representatives of these members are linearly independent. Let
$K_n$ denote $K_{1,n}$, \ie\ the closure of the three-braid
$(\beta_1)^n:=\left(\sigma_1 \sigma_2^{-1}\right)^n$ and let $J:=8_{17}\#8_{17}^r$. Recall that $8_{17}$ is negative-amphichiral and not concordant to its reverse~\cite{Kirk-Livingston:1999-2}.

\begin{thm}\label{thm:main2} The subgroup generated by $K_7, K_{11}, K_{17}, K_{23},$ and $J$ is isomorphic to $(\Z_2)^5 \leq \ker \varphi$.
\end{thm}

In general, using twisted Alexander polynomials to show that a fixed
knot $K$ is not slice is not so much technically difficult as
computationally intense.  Delaying all technical definitions to
Section~\ref{sec:twisted}, we say merely that in this context twisted
Alexander polynomials are associated to a choice of $q \in
\mathcal{Q}$ and a map $\chi \colon H_1(\Sigma_q(K);\Z) \to \Z_d$ for
some $d$. In order to use twisted Alexander polynomials to obstruct a
knot $K$ from being slice, one must show that for every subgroup $M$
of $H_1(\Sigma_q(K);\Z)$ satisfying certain algebraic properties there
exists a map $\chi$ vanishing on $M$ such that the resulting twisted
Alexander polynomial does not factor in a certain way.

By better understanding the structure of $H_1(\Sigma_q(K);\Z)$ one can sometimes significantly reduce the number of computations that are necessary. For example, Sartori's result of~\cite{Sartori} that $K_{7}$ is not slice requires the computation (and subsequent obstruction of factorization as a norm) of 170 different twisted Alexander polynomials, corresponding to order 13 characters vanishing on the 130 different square root order subgroups
of $H_1(\Sigma_7(K_{7});\Z)$. By careful consideration of the linking
form on $H_1(\Sigma_3(K_{n});\Z)$ and how its metabolizers are permuted
by the induced action of order $n$ symmetry of $K_{n}$, we are able to
prove that $K_{n}$ is not slice by computing only two twisted Alexander
polynomials, at least for $n=11,17,23$. In fact, while we do not
include these computations here, we leave as a challenge for the
interested reader to reprove Sartori's result by following roughly the
same argument below, but computing precisely 3 carefully chosen
twisted Alexander polynomials corresponding to $\chi \colon
H_1(\Sigma_3(K_{7});\Z) \to \Z_7$.

In addition, we overcome the following technical difficulty, which may be of independent interest. In many settings, the easiest way to  compute the homology of a knot's cyclic branched cover, with its linking form and module structure, is in terms of some nice Seifert surface.   However, the standard efficient algorithms for computing the twisted Alexander polynomial corresponding to $\chi \colon H_1(\Sigma_q(K);\Z) \to \Z_d$ require one to compute a map $\phi_\chi \colon \pi_1(X_K) \to GL(q, \Q(\xi_d)[t^{\pm1}])$ on the Wirtinger generators for $\pi_1(X_K)$. Relating these two perspectives is not entirely trivial, and we refer the reader to Appendix A for a discussion of this process.

%
%
%
%
\begin{rem} One can ask an analogous question in the topological category: Is there a knot that does not bound any topologically locally flat disk in the 4-ball but all its prime power fold cyclic branched covers bound topological rational homology balls? It turns out that such examples can be constructed by using the \emph{classical} Alexander polynomial. Let $\{n_i\}$ be the set of all natural numbers divisible by at least 3 distinct primes and $K_{i}$ be a knot with Alexander polynomial the $n_i^{th}$ cyclotomic polynomial. By Livingston \cite{Livingston:2002-1}, for each $i$, all the prime power fold cyclic branched covers along $K_i$ are integral homology spheres. Hence, by Freedman \cite{Freedman:1982-1, Freedman-Quinn:1990-1}, they all bound topological contractible four-manifolds. On the other hand, since the cyclotomic polynomials are irreducible, $K_i$ and $K_j$ are concordant if and only if $i=j$. Hence the knots $\{K_i\}$ represent distinct elements in $\ker \varphi^{top}$, the topological analogue of $\ker \varphi$.
\end{rem}

The results discussed in this introduction show that slice knots are not characterized by the property that each of their prime power fold cyclic branched covers bound rational homology balls.  However, there is a stronger condition that one might posit as a characterization of sliceness.  When a knot is slice, not only do its covers bound rational homology balls, but the deck transformations of the covers extend over these balls.  (Similarly, the lifts of the slice knot to knots in the covers bound slicing disks in these balls.)  This leads us to the following question.

\begin{quest}\label{q:extend}
	\ 
	\begin{enumerate}
		\item 	Does there exist a non-slice knot $K$ such that $\Sigma_q(K)$ bounds a rational homology ball for each prime power $q$ such that the deck transformations of $\Sigma_q(K)$ extend over the rational homology ball?
		\item 	Does there exist a non-slice knot $K$ such that $\Sigma_q(K)$ bounds a rational homology ball for each prime power $q$ such that the lift of $K$ to $\Sigma_q(K)$ bounds a disk in the rational homology ball?
	\end{enumerate}
\end{quest}

We remark that each of the knots $K_{m,n}$ studied in this article, as well as any knot of the form $K\#-K^r$ where $K$ is negative-amphichiral, can be shown to have the desired properties of Question~\ref{q:extend}(1) when $q$ is odd or the deck transformation is an involution, and the desired properties of Question~\ref{q:extend}(2) when $q$ is odd.

{Lastly, we make a remark on some other sliceness obstructions for $K_n$, where as above $n$ is an odd prime power not divisible by 3.
 Note that $K_n$ is strongly positive-amphichiral hence it is algebraically slice~\cite{Long:1984-1}. Further, $K_n$ is also strongly negative-amphichiral, which implies that it is rationally slice. Hence the $\tau$-invariant~\cite{Ozsvath-Szabo:2003-1}, $\varepsilon$-invariant~\cite{Hom:2014-1}, $\Upsilon$-invariant~\cite{Ozsvath-Stipsicz-Szabo:2017-1}, $\Upsilon^2$-invariant~\cite{Kim-Livingston:2018-1}, $\nu^+$-invariant~\cite{Hom-Wu:2016-1}, $\varphi_j$-invariants~\cite{DHST}, and $s$-invariant~\cite{Rasmussen:2010-1} all vanish for $K_n$. Moreover, since $[K_n] \in \ker \varphi$, the sliceness obstructions from the Heegaard Floer correction term and Donaldson’s diagonalization theorem (\eg \cite{Greene-Jabuka:2011-1, Jabuka:2012-1, Lisca:2007-1, Manolescu-Owens:2007-1}) applied to the cyclic branched covers of $K_n$ all vanish. As mentioned above, the fact that the involution induced by the deck transformation on $\Sigma_2(K_n;\Z)$ extends to a rational homology ball (in fact it is a $\mathbb{Z}_2$ homology ball) implies that sliceness obstructions such as \cite{Alfieri-Kang-Stipsicz:2019-1, Dai-Hedden-Mallick:2020-1} vanish.}

The paper is organized as follows: in Section~\ref{sec:knots} we prove Theorem~\ref{thm:main1}, and in Section~\ref{sec:twisted} we use twisted Alexander polynomials to show Theorem~\ref{thm:main2}.

\subsection*{Acknowledgements:}
This project began during a break-out session during the workshop
\emph{Smooth concordance classes of topologically slice knots} hosted
by the American Institute for Mathematics in June 2019.  The authors
would like to extend their gratitude to AIM for providing such a
stimulating research environment.  PA is supported by the European
Research Council (grant agreement No 674978). JM is supported by NSF
grant DMS-1933019. ANM is supported by NSF grant DMS-1902880.  MM is
supported by NSF grant DGE-1656466.  JP thanks Min Hoon Kim and Daniel
Ruberman for helpful conversations. AS was supported by the
\emph{\'Elvonal} Grant NKFIH KKP126683 (Hungary).  Lastly, we thank
Charles Livingston for pointing out the relevance of knots which are
not concordant to their reverses to this article.

\section{Branched covers bounding rational homology balls}
\label{sec:knots}

In this section, we will prove Theorem~\ref{thm:main1} after establishing  the following two propositions. We work in the smooth category.

\begin{prop}\label{prop:switch}
Suppose that $n$ and $q$ are both relatively prime to $2m+1$. Then~$\Sigma _q(K_{m,n})$ and $\Sigma _n (K_{m,q})$ are diffeomorphic
three-manifolds.
\end{prop}

\begin{proof}
We can realize $\Sigma_q(K_{m,n})$ by first taking the
$n$-fold cyclic branched cover  of $S^3$ branched along $A_{2m+1}$ and then the
$q$-fold cyclic branched cover branched along the pull-back of $B_{2m+1}$ of
Figure~\ref{fig:knot}. Since the roles of $A_{2m+1}$ and $B_{2m+1}$ are symmetric
(as shown by the left diagram of Figure~\ref{fig:knot}), this
three-manifold is the same as the $q$-fold cyclic branched cover branched
along $A_{2m+1}$, followed by the $n$-fold cyclic branched cover branched along 
the pull-back of $B_{2m+1}$, which is exactly $\Sigma _n (K_{m,q})$, concluding 
the argument.
\end{proof}

\begin{prop}\label{prop:ratslice}
Suppose that $n$ is relatively prime to $2m+1$. Then $K_{m,n}$ bounds a disk in a rational homology ball $X_{m,n}$ with only 2-torsion in $H_1(X_{m,n}; \Z )$.
\end{prop}

Recall that a knot is called \emph{rationally slice} if it bounds a smooth properly embedded disk in a rational homology ball and \emph{strongly negative-amphichiral} if there is an orientation-reversing involution~$\tau \colon S^3\to S^3$ such that $\tau (K)=K$  and the fixed point set of $\tau$ is a copy of $S^0\subset K$.

Proposition \ref{prop:ratslice} follows from the following lemma,
which is a special case of \cite{kawauchi}, together with a simple
observation regarding the knots $K_{m,n}$.

\begin{lem}[{\cite[Section 2]{kawauchi}}]
\label{lem:ratslice}
If $K$ is a strongly negative-amphichiral knot, then $K$ is slice in a rational homology ball $X$ with only 2-torsion in $H_1(X; \Z )$.
\end{lem}

\begin{proof}
Let $\tau$ be the orientation-reversing involution on $S^3$ with $\tau(K)=K$ where the fixed point set is two points. Let $M_K$ be the three-manifold obtained by performing 0-surgery on $K$. Then the involution $\tau$ extends from the exterior of $K$ to a fixed-point free orientation-reversing involution $\hat{\tau}$ on $M_K$.

The rational homology ball $X$ of the lemma is now constructed as follows: Consider the trace $W$ of the 0-surgery $M_K$, \ie\ $W$ is the four-manifold we get from $S^3\times [0,1]$ by attaching a 0-framed 2-handle along $K\subset S^3\times \{1\}$. Consider the quotient of $W$ by $\hat{\tau}$ on its boundary component diffeomorphic to $M_K$. The resulting compact four-manifold $X$ has $S^3$ as its boundary, and $K\subset S^3\times \{ 0\}$ is obviously slice in $X$: the slice disk is simply the core of the 2-handle (trivially extended through $S^3\times [0,1]$).

In order to complete the proof of the lemma, it would be enough to show that $H_*(X; \Q )=H_* (D^4; \Q)$ and $H_1(X; \Z ) \cong\mathbb{Z}_2$. For this computation, we consider an alternative description of $X$ as follows. Factoring $M_K$ by the free involution $\hat{\tau}$ we get a three-manifold $M$, together with a principal $\Z _2$-bundle $\pi \colon M_K \to M$ and an associated interval-bundle $Z\to M$. Note that $\partial Z = M_K$ and that $Z$ retracts to $M$. Then $X$ is the union of the surgery trace $W$ with $Z$, glued along $M_K$, \ie\ the four-manifold obtained by attaching 0-framed 2-handle along the meridian of $\partial Z = M_K$. The inclusion map $i$ induces the following exact sequence
$$H_1(\partial Z ;\mathbb{Z}) \xrightarrow{i_*} H_1(Z;\mathbb{Z}) \rightarrow \mathbb{Z}_2 \rightarrow 0.$$ This implies that $H_1(X; \Z ) \cong \mathbb{Z}_2$ since a 2-handle is attached along the generator of $H_1(\partial Z ;\mathbb{Z})$ to obtain $X$.\end{proof}

\begin{figure}[h]
\includegraphics[width=7cm]{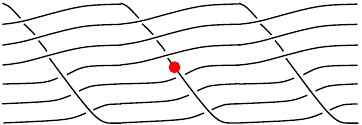}
\caption{Reflection to the red dot provides an involution $\tau \colon S^3\to
S^3$ verifying that the knot is strongly negative-amphichiral.}
\label{fig:sna}
\end{figure}

\begin{proof}[Proof of Proposition \ref{prop:ratslice}]
Figure~\ref{fig:sna} shows that $K_{m,n}$ is strongly negative-amphichiral; indeed, if the red dot of Figure~\ref{fig:sna}
is in the origin, the knot can be isotoped slightly so that the map
$v\mapsto -v$ for $v\in {\mathbb {R}} ^3$ provides the required $\tau$.  Then
Lemma~\ref{lem:ratslice} completes the proof of the
proposition.
\end{proof}

We recall a well known lemma of Casson and Gordon and for completeness
sketch its proof.

\begin{lem}[{\cite[Lemma 4.2]{Casson-Gordon:1978-1}}]\label{lem:cglemma}
Suppose that $q=p^\ell$ is an odd prime power, and $K$ is a knot that is slice in a rational homology ball $X$ with only $2$-torsion in $H_1(X;\mathbb{Z})$. Then $\Sigma_q(K)$ bounds a rational homology ball.\end{lem}

\begin{proof}
Let $D$ be the disk that $K$ bounds in $X$ and $\Sigma_{q}(D)$ be the $q$-fold cyclic branched cover of $X$ branched along $D$. Consider the infinite cyclic cover, denoted by $\widetilde{X}$, of $X \smallsetminus D$ and the following long exact sequence \cite{Milnor:1968-1} 
$$\dots \rightarrow \widetilde{H}_i(\widetilde{X};\mathbb{Z}_p)
\xrightarrow{t_*^{q}-\id} \widetilde{H}_i(\widetilde{X};\mathbb{Z}_p)
\rightarrow \widetilde{H}_i(\Sigma_{q}(D);\mathbb{Z}_p) \rightarrow
\widetilde{H}_{i-1}(\widetilde{X};\mathbb{Z}_p) \rightarrow \cdots$$
Here $t_*$ is the automorphism induced by the canonical covering
translation. Since $X$ is a rational homology ball with
only 2-torsion in the first homology, $t_*-\id$ is an
isomorphism. Moreover, with $\mathbb{Z}_p$ coefficients we have
$t_*^q-\id = (t_*-\id)^q $. Hence the result follows.
\end{proof}

\begin{proof}[Proof of Theorem~\ref{thm:main1}]
If  $q$ is an odd prime power, then  Proposition~\ref{prop:ratslice} and Lemma~\ref{lem:cglemma} together immediately imply that $\Sigma _q(K_{m,n})$ bounds a rational homology ball. 

Suppose now that $q=2^\ell$. By Proposition~\ref{prop:switch}, we have that  $\Sigma _q(K_{m,n})$ is diffeomorphic to $\Sigma _n (K_{m,q})$. Moreover $n$ was chosen to be an odd prime power,
while $q=2^\ell$ is relatively prime to $2m+1$. Hence the statement follows from 
the first case of this proof.
\end{proof}
%

%

\section{Sliceness obstructions from twisted Alexander polynomials}
\label{sec:twisted}
The goal of this section is to prove Theorem~\ref{thm:main2}. We first prove the following theorem, recalling that $K_n:=K_{1,n}$. 
\begin{thm}\label{thm:non-slice}
The knots $K_{11}, K_{17}$, and $K_{23}$ are not slice; hence are of
order two in ${\mathcal {C}}$.\end{thm}

The sliceness obstruction we intend to use in the proof of
Theorem~\ref{thm:non-slice} rests on a result of Kirk and Livingston
\cite{Kirk-Livingston:1999-1} involving twisted Alexander
polynomials. Throughout the rest of the section, $e^{2 \pi i/d}$ is
denoted by $\xi_d$, and the three-manifold obtained by performing
0-surgery on $K$ is denoted by $M_K$. We generally follow the
exposition of~\cite{HeraldKirkLivingston10}, and refer the reader to
that work for more details.

\begin{defn}\label{def:twstmod}
Given a representation $\alpha\colon \pi_1(M_K)\to
GL(q,\Q[\xi_d][t^{\pm 1}])$, the {\emph{twisted Alexander module}} $\mathcal{A}^\alpha(K)$ is the $\Q[\xi_d][t^{\pm 1}]$-module
$H_1(M_K;\Q[\xi_d][t^{\pm 1}]^q)$.\end{defn}

\begin{defn}\label{def:twstpoly}
The {\emph{twisted Alexander polynomial}} $\widetilde\Delta_K^\alpha(t)$ is the generator of the order ideal of $\mathcal{A}^\alpha(K)$; this polynomial is well-defined up to multiplication by units in $\Q[\xi_d][t^{\pm 1}]$.
\end{defn}
Twisted Alexander polynomials generalize the classical Alexander polynomial.  If we fix the representation $\alpha_0 \colon \pi_1(M_K) \to GL(1,\Q[t^{\pm 1}])$ (\ie $q=d=1$), then $\mathcal{A}^{\alpha_0}(K)$ is the classical (rational)
Alexander module $\mathcal{A}(K)$ of $K$ and
$\Delta_K(t):=\widetilde\Delta_K^{\alpha_0}(t)$ is the classical Alexander polynomial.

We will restrict to a special class of representations as
follows. First, choose $q \in \mathbb{N}$ and a character $\chi\colon
H_1(\Sigma_q(K);\Z)\to\Z_d$.  Note that $H_1(\Sigma_q(K);\Z)\cong
\mathcal{A}(K)/\langle t^q-1 \rangle$ and that a choice of a meridian
for $K$ determines a map from $\pi_1(M_K)$ to $\Z \ltimes
\mathcal{A}(K)/\langle t^q-1 \rangle$, as discussed in more detail in
Appendix~\ref{section:polycomp}. The character $\chi$ therefore
induces $\alpha_\chi\colon \pi_1(M_K)\to GL(q,\Q[\xi_d][t^{\pm 1}])$,
and we write
$\widetilde\Delta_K^\chi(t):=\widetilde\Delta_K^{\alpha_\chi}(t)$.
This is a very quick explanation of twisted Alexander
polynomials, and Friedl and Vidussi \cite{friedlvidussi} have a survey of
twisted Alexander polynomials which we recommend for more detailed
exposition.

The obstruction we will use in the proof of
Theorem~\ref{thm:non-slice} is a generalization of the Fox-Milnor
condition~\cite{FoxMilnor}, which states that the Alexander polynomial
of a slice knot factors as $f(t)f(t^{-1})$ for some $f(t) \in
\Z[t^{\pm1}]$. First, recall the following definition.

\begin{defn}\label{def:norm} We call a Laurent polynomial $d(t)\in \Q(\xi_d)[t^{\pm1}]$ a \emph{norm} 
if there exist $c \in \Q(\xi_d)$, $k \in \Z$, and $f(t)
\in \Q(\xi_d)[t^{\pm1}]$ such that
$$d(t)= c t^k f(t) \overline{f(t)},$$ where $\overline{\,\cdot\,}$ is
induced by the $\Q$-linear map on $\Q(\xi_d)[t^{\pm 1}]$ sending $t^i$
to $t^{-i}$ and $\xi_d$ to $\xi_d^{-1}$.
\end{defn}

\begin{thm}[\cite{Kirk-Livingston:1999-1}]\label{thm:twistedalexsliceobstruction}
Suppose that $K\subset S^3$ is a slice knot and $q$ is a prime power. Then there exists a covering transformation invariant metabolizer $P \leq
H_1(\Sigma_q(K); \Z )$ such that if $$\chi \colon H_1(\Sigma_q(K); \Z )
\to \Z_d$$ is a character of odd prime power order such that
$\chi|_P=0$, then $\widetilde{\Delta}_K^{\chi}(t)\in \Q(\xi_d)[t^{\pm1}]$ is a norm.  \qed\end{thm}

Let $K \in \{K_{11}, K_{17}, K_{23} \}$. We first determine the metabolizers of $H_1(\Sigma_3(K);\Z)$ and construct prime order characters vanishing on
each metabolizer in Subsection~\ref{section:homologyofcover}.  We then
show that the corresponding twisted Alexander polynomials of $K$ do
not factor as a norm in Section~\ref{sec:proof}.

\subsection{The metabolizers of $H_1(\Sigma_3(K_n);\Z)$}
\label{section:homologyofcover}


We assume that
$n$ is odd and not divisible by 3, so in particular $K_n$ is a knot. Our understanding of $H_1(\Sigma_3(K_n);\Z)$ and its metabolizers will
come from a computation of the Alexander module and the Blanchfield
pairing of $K_n$.  Throughout this section, we also keep track of the
order $n$ symmetry of $K_n$, which will be useful later on to reduce
the number of twisted Alexander polynomials we must compute.

Observe that $K:=K_n$ has a genus $n-1$ Seifert surface $F$, illustrated in Figure~\ref{fig:seif} for $n=7$,  which is
invariant under the periodic order $n$ symmetry $r \colon S^3 \to S^3$ given diagrammatically by rotating counterclockwise by $2\pi/n$.
\begin{figure}[h]
\includegraphics[height=5cm]{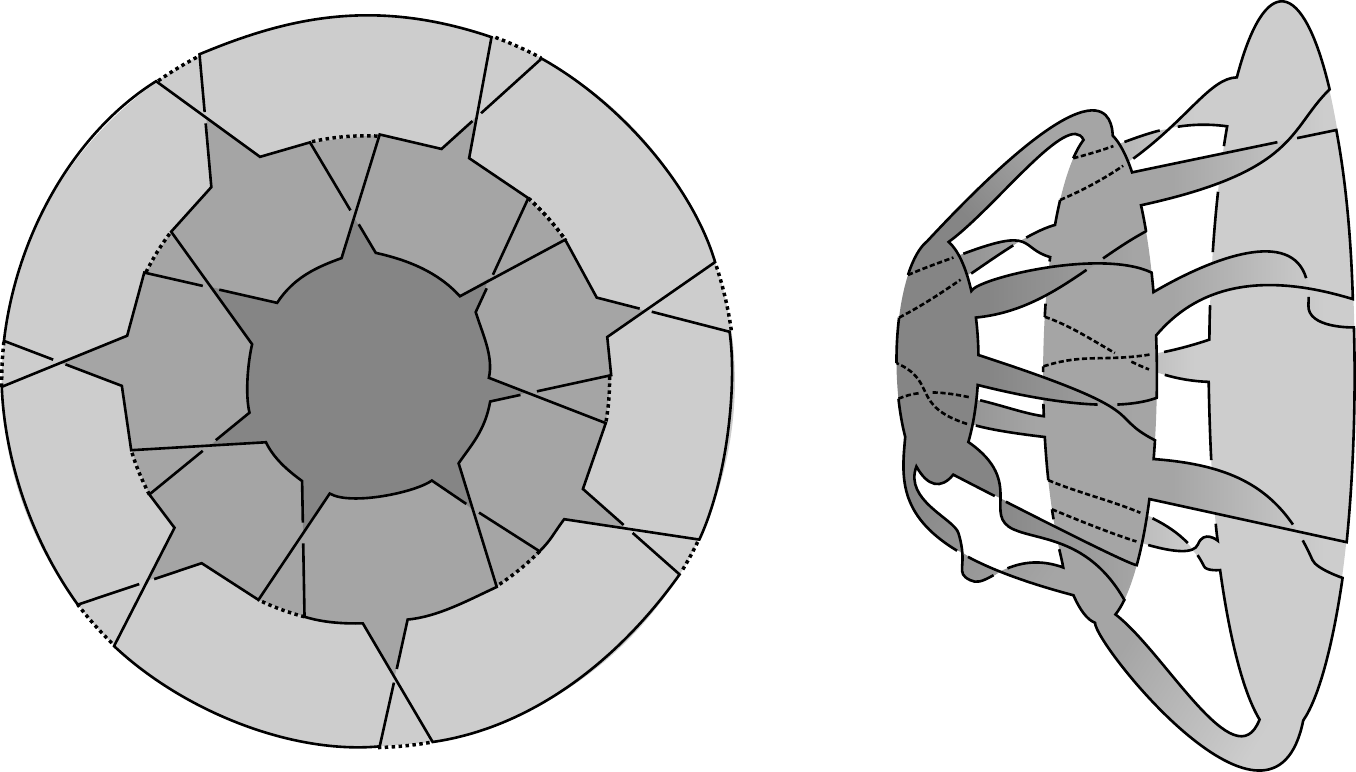}
\caption{A Seifert surface $F$ for $K$ from two different perspectives.} \label{fig:seif}
\end{figure}
We pick a collection of simple closed curves $\al_1, \dots, \al_{n-1}, \be_1, \dots, \be_{n-1}$ on $F$ that form a basis for $H_1(F; \Z)$ as illustrated in Figure~\ref{fig:basis}.
\begin{figure}[h]
\includegraphics[height=4cm]{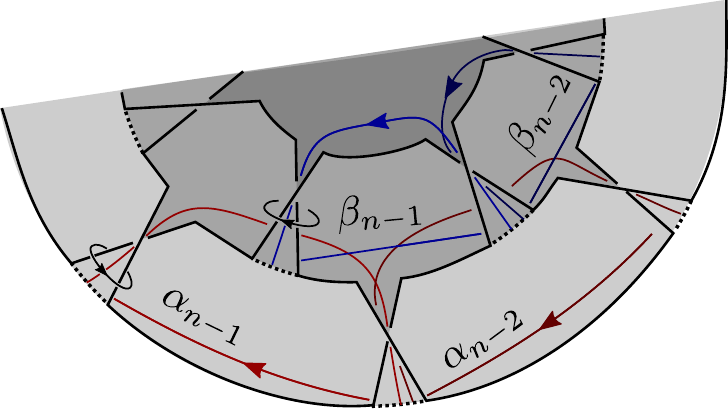}
\caption{A basis of curves for $H_1(F; \Z)$.} \label{fig:basis}
\end{figure}
 Note that $r(\al_i)=\al_{i-1}$ and $r(\be_{i})= \be_{i-1}$ for $i>1$, while the induced action of $r$ on $[\al_1], [\be_1] \in H_1(F; \Z)$ is given by 
$$r_*([\al_1])=\sum_{i=1}^{n-1}
-[\al_{i}] \text{ and } r_*([\be_1])=\sum_{i=1}^{n-1} -[\be_{i}].$$
%

It is straightforward to compute the Seifert matrix $A$ for the Seifert pairing on $F$ with respect to our fixed basis, and we obtain
$A= \left[ \begin{array}{cc} -B^T& 0 \\ B & B \end{array}\right],$
where $B$ is the $(n-1) \times (n-1)$ matrix with entries given by
$
B_{i,j}= \begin{cases} 1 & i=j \\ -1 & i=j-1 \\ 0 & \text{else} \end{cases}
$.  Recall that Blanchfield \cite{Blanchfield:1957-1} showed that the Alexander module $\A(K)$ supports a non-singular pairing $$\Bl\colon \A(K) \times\A(K) \to \Q(t)/ \Z[t^{\pm1}]$$ called the \emph{Blanchfield pairing}. The pairing can be computed using a Seifert matrix of $K$ as follows, for more details see \cite{Friedl-Powell:2017-1, Kearton:1975-1, Levine:1977-1}.

\begin{thm}[{\cite[Theorem 1.3 and 1.4]{Friedl-Powell:2017-1}}]\label{thm:FriedlPowell}
Let $F$ be a Seifert surface for a knot $K$ with a collection of simple closed curves $\delta_1, \dots, \delta_{2g}$ on $F$ that form a basis for $H_1(F; \Z)$ and corresponding Seifert matrix $A$.
Let $\widehat{\delta}_1, \dots,\widehat{\delta}_{2g}$ be a collection
of simple closed curves in $S^3 \smallsetminus \nu(F)$ representing a
basis for $H_1(S^3 \smallsetminus \nu( F); \Z)$ satisfying 
$\lk(\delta_i, \widehat{\delta_j})= \delta_{i,j}$ $($\ie the Alexander
dual basis$)$, where $\nu(F)$ denotes an open tubular neighborhood $F\times I$. Consider the standard decomposition of the infinite
cyclic cover of the knot exterior as
\[
X_K^{\infty} = \bigcup_{i=-\infty}^{+\infty} (S^3 \smallsetminus \nu(F))_i,
\]
and let the homology class of the unique lift of $\widehat{\delta}_i$
to $(S^3 \smallsetminus \nu(F))_0$ be denoted by $d_i$. Then the map
\begin{align*}
p \colon \left(\Z[t^{\pm1}]\right)^{2g} & \to \A(K) \\
(x_1, \ldots, x_{2g}) & \mapsto \sum_{i=1}^{2g} x_i d_i.
\end{align*} is surjective and has kernel given by
$(tA-A^T)\Z[t^{\pm{1}}]^{2g}$. Moreover, the Blanchfield pairing is given as follows: for $x,y \in \Z[t^{\pm1}]^{2g}$ we have $$\Bl(p(x),p(y))=(t-1)x^T (A-tA^T)^{-1} \overline{y} \in \Q(t)/ \Z[t^{\pm1}],$$ where $\overline{\,\cdot\,}$ is induced by the $\Z$-linear map on
$\Z[t^{\pm1}]$ sending  $t^i$ to $t^{-i}$.\qed
\end{thm}

Following the language above, let $\hat{\al}_1, \dots,
\hat{\al}_{n-1}, \hat{\be}_1, \dots, \hat{\be}_{n-1}$ be the Alexander
dual basis of $\al_1, \dots, \al_{n-1}, \be_1, \dots, \be_{n-1}$ and
$a_i, b_i$ be the homology classes of the unique lifts of
$\hat{\al}_i, \hat{\be}_i$, respectively. Note that $\hat{\al}_{n-1}$
and $\hat{\be}_{n-1}$ are illustrated in Figure~\ref{fig:basis} as
small closed curves linking $F$. By inspecting the matrix $tA-A^T$,
illustrated below for $n=7$,
\[
\left[\begin{array}{cccccc|cccccc}
1-t & t & 0 & 0 & 0 & 0 & -1 & \bf{1} & 0 & 0 & 0 & 0 \\ 
-1 & 1-t &t& 0 & 0 & 0 & 0  & -1 &  \bf{1}  & 0 & 0 & 0 \\ 
0 & -1& 1-t & t  & 0 & 0 & 0 & 0 & -1 &  \bf{1}  & 0 & 0 \\ 
0 & 0 & -1 & 1-t &t & 0 & 0 & 0 & 0 & -1 &  \bf{1}  & 0 \\ 
0 & 0 & 0 & -1 & 1-t & t & 0 & 0 & 0 & 0 & -1 &  \bf{1}  \\ 
0 & 0 & 0 & 0 & -1 & 1-t & 0 & 0 & 0 & 0 & 0 & -1 \\ 
\hline
t & 0 & 0 & 0 & 0 & 0 & t-1 & 1 & 0 & 0 & 0 & 0 \\ 
-t & t & 0& 0 & 0 & 0 & -t & t-1 &1 & 0 & 0 & 0 \\ 
0 & -t  & t & 0 & 0 & 0 & 0 &  -t & t-1 & 1& 0 & 0 \\ 
0 & 0 &-t  & t & 0 & 0 & 0 & 0 &  -t & t-1 & 1 & 0 \\ 
0 & 0 & 0 & -t  & t & 0 & 0 & 0 & 0 &  -t & t-1 & 1 \\ 
0 & 0 & 0 & 0 & -t  & t & 0 & 0 & 0 & 0 & -t & t-1
\end{array} 
\right]
\]
 we see that we can use the bolded pivot entries to perform column operations over $\Z[t^{\pm1}]$ to transform $tA-A^T$ to a matrix as below:
\[\left[
\begin{array}{cccccc|cccccc}
 0 & 0 & 0  & 0 & 0 & 0 & 0 & 1 & 0 & 0 & 0 & 0 \\ 
 0 & 0 & 0 & 0 & 0 & 0 & 0  &  0 & 1   & 0 & 0 & 0 \\ 
0 &  0 & 0 & 0   & 0 & 0 & 0 & 0 & 0 & 1   & 0 & 0 \\ 
0 & 0 & 0 & 0 & 0 & 0 & 0 & 0 & 0 &  0 & 1   & 0 \\ 
0 & 0 & 0 &  0 & 0 & 0  & 0 & 0 & 0 & 0 &  0 & 1   \\ 
* & * & * & * &  * & * &* & * &* & * & * & * \\ 
\hline
* & \bf{-t} & 0 & 0 & 0 & 0 & t & 1 & 0 & 0 & 0 & 0 \\ 
* & * & \bf{-t}& 0 & 0 & 0 & 0& t &1 & 0 & 0 & 0 \\ 
*& * & * &\bf{-t}& 0 & 0 & 0 &   0& t &1 & 0 & 0 \\ 
*& * &* & * & \bf{-t}  & 0 & 0 & 0 &   0& t &1  & 0 \\ 
*& * &* & * & * & \bf{-t} & 0 & 0 & 0 &   0& t &1  \\ 
*& * &*& * & * & * &  -1 &  -1&  -1 & -1 & -1 & t-1
\end{array} 
\right].
\]
We now use the new bolded entries as pivots to perform column operations to obtain a matrix  whose $i^{th}$ row has a single non-zero entry that occurs in column $i+1$, for all  $i=1, \dots, n-2, n, \dots, 2n-3$. This  matrix is of the following form: 
 \[\left[
\begin{array}{cccccc|cccccc}
 0 & 0 & 0 & 0 & 0 & 0 & 0 & 1 & 0 & 0 & 0 & 0 \\ 
0 & 0 &0 & 0 & 0 & 0 & 0 & 0 & 1 & 0 & 0 & 0 \\ 
0 & 0 & 0 & 0 & 0 & 0 & 0 & 0 & 0 & 1 & 0 & 0 \\ 
0 & 0 & 0 & 0 & 0 & 0 & 0 & 0 & 0 & 0 & 1 & 0 \\ 
0 & 0 & 0 & 0 & 0 & 0 & 0 & 0 & 0 & 0 & 0 & 1\\ 
 *_{n-1,1} & *  & *  &  * &   *&   *& *_{n-1,n} & *  &  * &  * &   *&   *\\ 
\hline
 0& -t & 0 & 0 & 0 & 0 & 0 & 0 & 0 & 0 & 0 & 0 \\ 
0 & 0 & -t & 0 & 0 & 0 & 0 & 0 & 0 & 0 & 0 & 0 \\ 
0 & 0 & 0 & -t  & 0 & 0 & 0 & 0 & 0 & 0 & 0 & 0 \\ 
0 & 0 & 0 & 0 & -t  & 0 & 0 & 0 & 0 & 0 & 0 & 0 \\ 
0 & 0 & 0 & 0 & 0 & -t & 0 & 0 & 0 & 0 & 0 & 0 \\ 
 *_{2n-2, 1} & *  & *  & *  & *  & *  & *_{2n-2, n}&*   &*   &  * &  * &  *
\end{array} 
\right].
\]
Notice that only the $*$-entries with indices have an impact on
$\A(K)$.  In particular, $\A(K)$ is generated by $a_{n-1}$ and
$b_{n-1}$, in the language of the notation introduced just after Theorem~\ref{thm:FriedlPowell}.

For $n=7,11,17,23$ one continues to perform column moves until the above matrix is
simplified to the following form:
\[\left[
\begin{array}{cccccc|cccccc}
0 & 1 & 0 & 0 & 0 & 0 & 0 & 0 & 0 & 0 & 0 & 0 \\ 
0 & 0 & 1 & 0 & 0 & 0 & 0 & 0 & 0 & 0 & 0 & 0 \\ 
0 & 0 & 0 & 1 & 0 & 0 & 0 & 0 & 0 & 0 & 0 & 0 \\ 
0 & 0 & 0 & 0 & 1 & 0 & 0 & 0 & 0 & 0 & 0 & 0 \\ 
0 & 0 & 0 & 0 & 0 & 1 & 0 & 0 & 0 & 0 & 0 & 0 \\ 
p_n(t) & *  & *  & *  &*   &  * & 0 & *  &  * &  * &  * &  * \\ 
\hline
 0& 0 & 0 & 0 & 0 & 0 & 0 & 1 & 0 & 0 & 0 & 0 \\ 
0 & 0 & 0 & 0 & 0 & 0 & 0 & 0 & 1 & 0 & 0 & 0 \\ 
0 & 0 & 0 & 0 & 0 & 0 & 0 & 0 & 0 & 1 & 0 & 0 \\ 
0 & 0 & 0 & 0 & 0 & 0 & 0 & 0 & 0 & 0 & 1 & 0 \\ 
0 & 0 & 0 & 0 & 0 & 0 & 0 & 0 & 0 & 0 & 0 & 1 \\ 
0 & *  & *  &*   &*   & *  & p_n(t) &*   & *  &   *&  * &  *
\end{array} 
\right],
\]
where $$p_n(t)=
\prod_{k=0}^{(n-1)/2}\left(t^2+(\xi_n^k-1+\xi_n^{-k})t+1\right).$$
This and all further computations in
Section~\ref{section:homologyofcover} were done in a Jupyter notebook
and is available on the third author's website.  In particular, this
implies that $\Delta_{K_n}(t)= p_n(t)^2$, which one can verify for
general $n \in \mathbb{N}$ by using the formula for the Alexander
polynomial of a periodic knot in terms of the multivariable Alexander
polynomial of the quotient link~\cite{murasugiperiodic}.

Using the above matrix, we obtain for our values of interest that 
\[\A(K) \cong \Z[t^{\pm1}]/ \langle p_n(t) \rangle \oplus \Z[t^{\pm1}]/ \langle p_n(t) \rangle,\]
where the two summands are respectively generated by $a:=a_{n-1}$ and $b:=b_{n-1}$.

We can also compute the action induced by the order $n$ symmetry $r$ on $\A(K)$. In particular, we can observe that $r(\widehat{\al}_{n-1})$ is a curve whose only non-trivial linkage is $-1$ with $\al_{n-1}$ and $+1$ with $\al_{n-2}$. Similar observations can be made for $r(\widehat{\be}_{n-1})$, and so it follows that the induced action of $r$ on $[\widehat{\al}_{n-1}], [\widehat{\be}_{n-1}] \in H_1(S^3 \smallsetminus \nu(F); \Z)$ is given by
 $$r_*([\widehat{\al}_{n-1}])=-[\widehat{\al}_{n-1}]+[\widehat{\al}_{n-2}] \text{ and } r_*([\widehat{\be}_{n-1}])=- [\widehat{\be}_{n-1}] +[ \widehat{\be}_{n-2}].$$
 Therefore, the action of $r_*$ on the generators of $\A(K)$ is given by 
 $$r_*(a_{n-1})= -a_{n-1}+ a_{n-2} \text{ and } r_*(b_{n-1})= -b_{n-1}+ b_{n-2}.$$ Moreover, by considering the $(n-1)^{th}$ and $(2n-2)^{th}$ columns of $tA-A^T$, we obtain the relations
$$t a_{n-2} + (1-t) a_{n-1}+ t b_{n-1}=0,$$
$$a_{n-2} -a_{n-1} + b_{n-2} +(t-1) b_{n-1} =0.$$
Simple algebraic manipulations give us that 
\begin{equation}\label{eq:actiona}
r_*(a)=r_*(a_{n-1}) =-a_{n-1}+ a_{n-2}
= -t^{-1}a- b,\end{equation}
\begin{equation}\label{eq:actionb}r_*(b)=r_*(b_{n-1})=-b_{n-1}+b_{n-2}
=t^{-1} a +(1- t)b.\end{equation}
Moreover, we obtain that  if $v=f_1(t)a + g_1(t)b$ and $w= f_2(t)a +g_2(t)b$ then
\begin{align*}
\Bl(v,w)=\left[ \begin{array}{c}
f_1(t)\\ g_1(t)
\end{array}\right] ^T\cdot
 \left[ \begin{array}{cc} 
c_{11}& c_{12} \\
c_{21} & c_{22}
\end{array}\right]
\cdot
\left[
\begin{array}{c}
f_2(t^{-1})\\ g_2(t^{-1})
\end{array}
\right]
\end{align*}
where $c_{ij}=  (t-1) (A-tA^T)^{-1}_{(i(n-1),j(n-1))}$. We remark that the interested reader can use this formula to algebraically verify the
geometrically immediate fact that $\Bl(r_*(v), r_*(w))= \Bl(v,w)$ for all
$v,w \in \A(K)$.

In applying Theorem~\ref{thm:twistedalexsliceobstruction} we will take
$q=3$, that is, we will consider the 3-fold cyclic branched cover $\Sigma
_3(K)$ of $S^3$ branched along $K$, and will derive the sliceness obstruction
from that cover.  We wish to transfer our information about
$(\A(K), \Bl)$ to tell us about $(H_1(\Sigma_3(K); \Z), \lambda)$.  First,
we have that
\begin{align*}
 H_1(\Sigma_3(K); \Z)& \cong \A(K) / \langle t^2+t+1 \rangle \\
 &\cong 
\Z[t^{\pm1}]/ \langle p_n(t), t^2 + t+1 \rangle  \oplus  \Z[t^{\pm1}]/ \langle p_n(t), t^2 + t+1 \rangle   \\
&\cong\Z_n[t^{\pm1}]/ \langle t^2 + t+1 \rangle \oplus \Z_n[t^{\pm1}]/ \langle t^2 + t+1 \rangle,
\end{align*}
where the two summands are generated by the images of $a$ and $b$
(equivalently, lifts of the homology classes of the curves $\widehat{\al}_{n-1}$
and $\widehat{\be}_{n-1}$ to the preferred copy of $S^3 \smallsetminus
\nu(F)$ in $\Sigma_3(K)$). In particular, as a group $H_1(\Sigma_3(K); \Z) \cong (\Z_n)^4$, with natural generators the images of $a, ta, b,$ and $tb$. By a mild abuse of notation, we blur the distinction between the elements of the Alexander module and corresponding elements of $H_1(\Sigma_3(K); \Z)$.

The following result, which is slightly reformulated from~\cite{Friedl:2003-1}, lets us compute the torsion linking form $\lambda$ with respect to our preferred basis. 

\begin{prop}[{\cite[Chapter 2.6]{Friedl:2003-1}}]\label{prop:stefanthesis}
Suppose that $q$ is a prime power and let
$x, y \in H_1(\Sigma_q(K); \Z)$. Choose $\tilde{x}, \tilde{y} \in \A(K)$ which lift $x$ and $y$,
and write $$\Bl(\tilde{y},\tilde{x})= \frac{p(t)}{\Delta_K(t)} \in \Q(t)/
\Z[t^{\pm1}].$$ Since $t^q-1$ and $\Delta_K(t)$ are
relatively prime, one can find $r(t) \in \Z[t^{\pm1}]$ and $c \in
\mathbb{Z}$ such that $\Delta_K(t)r(t) \equiv c \pmod{t^q-1}$. 
Writing $p(t)r(t) \equiv \sum_{i=1}^{q} \alpha_i t^i \pmod{t^q-1}$, 
for $i=0, \dots ,q-1$ we obtain
\[\pushQED{\qed}\lambda_q(x, t^{i} y)= \frac{\alpha_{q-i}}{c}\in \Q/ \Z.
\qedhere\]
\end{prop}
 
From now on, we take $n$ to be $11$, $17$, or $23$. We expect that the subsequent computations of this section will  hold for general $n\equiv 5\pmod{6}$, but we have not verified these results for $n>23$. When we apply this process to our formula for $\Bl$, we obtain that with respect to the $\Z_n$-basis $\{a, ta, b, tb\}$ our linking form
is given by the matrix
\begin{align*}
L= \frac{1}{n}
 \left[
\begin{array}{cccc}
-1 & -k & -k & k \\
-k & -1 & 0 & -k \\
-k & 0 & 1 & k\\
k & -k & k & 1
\end{array}
\right],
\end{align*}
where $n=2k+1$.

We now wish to show that there are exactly two orbits of the action of
$r$ on the collection of invariant metabolizers of $H_1(\Sigma_3(K_n);\Z)$; this will imply later on that the computation of two twisted Alexander polynomials will suffice to obstruct the sliceness of $K_n$. 
Note that our formulas $(\ref{eq:actiona})$ and $(\ref{eq:actionb})$ hold equally well for the induced action of
$r$ on $H_1(\Sigma_3(K); \Z)$, once we apply the relation $t^3=1$. 
Recalling that $n\in\{11,17,23\}$, we note that since $n \equiv 5 \pmod{6}$ the polynomial $t^2+t+1$ is irreducible in $\Z_n[t^{\pm1}]$. Therefore, since $n$ is also a prime, we see that $\Z_n[t^{\pm1}]/ \langle t^2+t+1 \rangle$ has no non-trivial proper submodules. It follows that there are exactly $n^2+1$ order $n^2$ submodules of $H_1(\Sigma_3(K); \Z)$: first, for any $n_0,n_1 \in \Z_n$ we have \[P_{n_0, n_1}:= \spa_{\Z_n[t^{\pm1}]}\{a+ (n_0+n_1 t)b\}
= \spa_{\Z_n}\{ a+ n_0b + n_1 tb, ta- n_1 b + (n_0-n_1) tb\}
\]
and secondly we have $$P':=\spa_{\Z[t^{\pm1}]}\{ b\}= \spa_{\Z}\{b, tb\}.$$
Using the matrix $L$, we see that  $\lambda(b,b)= \frac{1}{n}  \neq 0\in \Q/\Z,$ and so $P'$ is not a metabolizer. Moreover, observe that the condition $$\lambda(a+ (n_0+n_1 t)b, a+ (n_0+n_1 t)b)=0 \in \Q/\Z$$ gives us a 2-variable ($n_0$ and $n_1$) quadratic polynomial  over $\Z_n$, and hence has at most $2n$ solutions. 

Letting  $\mathcal{P}$ denote the set of all metabolizers, we have shown that $$| \mathcal{P}| \leq 2n .$$ Moreover, note that the map $r_*$ acts on $\mathcal{P}$ and since $n$ is prime and $(r_*)^n=Id$, the orbit of a metabolizer is either of order $n$ or 1. 

A short algebraic argument shows that $r_*(P_{n_0, n_1})= P_{n_0, n_1}$ if and only if $n_0= n_1= 1$. The `if' direction follows immediately from Equation~\eqref{eq:actiona} and~\eqref{eq:actionb}. For the `only if' direction, compute
\[ r(a+ n_0b + n_1 tb)= (1-n_0+n_1)a+(1-n_0)ta+(-1+n_0+n_1) b+ (-n_0+2n_1)tb
\]
and observe that if this element belongs to $P_{n_0, n_1}$ then by looking at the $a$ and $ta$ coefficients we see that  it must equal 
\[(1-n_0+n_1) (a+ n_0b + n_1 tb) + (1-n_0)(ta- n_1 b + (n_0-n_1) tb).\]
Contemplation of the coefficients of $b$ and $tb$ in these two expressions shows that they can only be equal if $n_0=n_1=1$. Moreover, it is not hard to explicitly verify that $P_{-1,-1}$ is also a metabolizer and so there are exactly two orbits. We choose a representative metabolizer for each orbit: 
 \begin{align}\label{eqn:orbitmet}  
 P_+:=P_{1,1} =\spa_{\Z} \{a+ b+tb, ta -  b \} \text{ and  } P_-:=P_{-1,-1}= \spa_{\Z} \{a-b-tb, ta +  b \}.
 \end{align} 
  We note for future reference that it is extremely easy to construct a character $$\chi \colon H_1(\Sigma_3(K); \Z) \to \Z_n$$ vanishing on $P_{\pm}$: choose $\chi(b)$ and $\chi(tb)$ freely and $\chi(a)$ and $\chi(ta)$ are determined.  In fact, we  choose $\chi_\pm$ as follows:
\begin{align}\label{eqn:chidef}
\chi_\pm(a)= \pm1, \chi_\pm(ta)= 0, \chi_\pm(b)=0, \text{ and } \chi _\pm(tb)=-1.
\end{align}
To avoid confusion, we point out here that the `$d$' of Definitions~\ref{def:twstmod} and~ \ref{def:twstpoly} and Theorem~\ref{thm:twistedalexsliceobstruction} happens to be $n$ for us.

\subsection{Proof of the main theorems}\label{sec:proof}

To apply Theorem~\ref{thm:twistedalexsliceobstruction}, we must obstruct the existence of certain factorizations in $\Q(\xi_d)[t^{\pm1}]$. It is easier to obstruct the existence of factorizations in $\Z_p[t^{\pm1}]$, where computer programs are for finiteness reasons capable of proving that no factorization of a given kind exists, and the following propositions allow us to make this transition. 

\begin{prop}[{\cite[Lemma 8.6]{HeraldKirkLivingston10}}]\label{prop:HKLreduction}Let $d, s$ be primes and suppose $s=kd+1$. Choose $\theta \in \mathbb{Z}_s$ so that $\theta \in \mathbb{Z}_s$ is a primitive $d^{th}$ root of unity modulo $s$. The choice of $s$ and $\theta$ defines a map $\pi \colon \Z[\xi_d][t^{\pm1}] \to
\Z_s[t^{\pm1}]$ where $1$ is mapped to $1$ and $\xi_d$ is mapped to $\theta$.

Let $d(t)\in \Z[\xi_d][t^{\pm1}]$ be a polynomial of degree $2N$ such that  $\pi(d(t)) \in \Z_s[t^{\pm1}]$ also has degree $2N$. If $d(t)\in \Q(\xi_d)[t^{\pm1}]$ is a norm then $\pi(d(t)) \in \Z_s[t^{\pm1}]$ factors as the product of two polynomials of degree $N$.\qed\end{prop}

\begin{prop}\label{prop:reduction}
Given a knot $K$, a preferred meridian $\mu_0$, and a map $\chi \colon H_1(\Sigma_q(K); \Z) \to \Z_d$ where $d$ is a prime, we obtain as above a reduced twisted Alexander polynomial $\tilde{\Delta}_K^{\chi}(t)$. 
By rescaling, assume that $\tilde{\Delta}_K^{\chi}(t)$ is an element of $\Z[\xi_d][t^{\pm1}]$.

Let $s=kd+1$, $\theta \in \mathbb{Z}_s$, and    $\pi \colon \Z[\xi_d][t^{\pm1}] \to
\Z_s[t^{\pm1}]$ be as in Proposition~\ref{prop:HKLreduction}. Suppose that $\pi\left( \widetilde{\Delta}_K^{\chi}(t)\right)$ is a degree $2 \lfloor\frac{c(K)-3}{2}\rfloor$ polynomial which cannot be written
as a product of two degree $\lfloor\frac{c(K)-3}{2}\rfloor$
polynomials in $\Z_s[t^{\pm1}]$. Then $\tilde{\Delta}_K^{\chi}(t) \in \Q(\xi_d)[t^{\pm1}]$ is not a norm.
\end{prop}

Here, degree is taken to be the degree of a Laurent polynomial -- \ie $\deg_{\max}-\deg_{\min}$. Proposition \ref{prop:reduction} is useful for efficient computations, since in our setting
$\det(\phi_\chi(g_1))=t-1$ and one can compute
\begin{align*} \label{eq:comppoly}
\pi\left(\widetilde{\Delta}_K^{\chi}(t)\right)= 
\frac{ \det\Big( \left[\pi \left( \Phi\left( \frac{\partial r_i}{\partial g_j}\right) \right)\right]_{i,j=2}^c\Big)}{(t-1)^2},
\end{align*}
in particular, computing determinants of matrices with entries in
$\Z_s[t^{\pm 1}]$ rather than in $\Q[\xi_d][t^{\pm1}]$.

\begin{proof}[Proof of Proposition~\ref{prop:reduction}]
By Proposition~\ref{prop:HKLreduction}, to establish our desired result under the above hypotheses it suffices to show that the degree of 
$\widetilde{\Delta}_K^{\chi}(t)$ is equal to $2 \lfloor\frac{c(K)-3}{2}\rfloor$, \ie that the reduced twisted Alexander polynomial does not drop degree under $\pi$. 
By considering Proposition~\ref{prop:computation} and recalling that we choose $\phi_\chi(g_1)$ to have determinant equal to $t-1$, we see that 
the degree of $\widetilde{\Delta}_K^{\chi}(t)$ is no more than $c(K)-3$ as follows. 

The degree of $\widetilde{\Delta}_K^{\chi}(t)$ is 2 less than the degree of  $ \det\Big( \left[\pi \left( \Phi\left( \frac{\partial r_i}{\partial g_j}\right) \right)\right]_{i,j=2}^c\Big)$. 
The Wirtinger presentation of  $\pi_1(X_K)$ has $c(K)$ generators and $c(K)$ relations of the form $r_i=g_{a_i} g_{b_i} g_{c_i}^{-1} g_{b_i}^{-1}$ for some $a_i,b_i,c_i$. Moreover, since $g_{a_i} g_{b_i} g_{c_i}^{-1} g_{b_i}^{-1}=1$ one can verify that
\[ \partial(g_{a_i} g_b g_c^{-1} g_b^{-1})
= \partial((g_{a_i} g_{b_i}) (g_{b_i} g_{c_i})^{-1}) 
= \partial(g_{a_i} g_{b_i})- \partial(g_{b_i} g_{c_i})
= \partial(g_{a_i}) +(g_{a_i}-1) \partial(g_b) - g_{b_i} \partial(g_{c_i}).\]
Therefore for any $i,j$ we have that  
 \begin{align*}
  \Phi\left( \frac{ \partial r_i }{\partial g_j}\right)=
\begin{cases}
\left[\begin{array}{ccc}
 1& 0 & 0 \\ 0 & 1 &0 \\ 0 & 0 & 1 \end{array}
\right]
 & \text{ if } j= a_i,\\
\vspace{-0.4cm}\\ 
\left[\begin{array}{ccc}
 0& 0 & t \\ 1 & 0 &0 \\ 0 & 1 & 0 \end{array}
\right]
\left[\begin{array}{ccc}
 \xi_d^{*} & 0 & 0 \\ 0 & \xi_d^{**} &0 \\ 0 & 0 & \xi_d^{***} \end{array}
\right]
-\left[\begin{array}{ccc}
1 & 0 & 0 \\ 0 &1 &0 \\ 0 & 0 & 1\end{array}
\right] & \text{ if } j= b_i,\\
\vspace{-0.4cm}\\ 
\left[\begin{array}{ccc}
 0& 0 & t \\ 1 & 0 &0 \\ 0 & 1 & 0 \end{array}
\right]
\left[\begin{array}{ccc}
 \xi_d^{*} & 0 & 0 \\ 0 & \xi_d^{**} &0 \\ 0 & 0 & \xi_d^{***} \end{array}
\right] & \text{ if }j=c_i,
\end{cases}
 \end{align*}
 and is the $3 \times 3$   zero matrix if $ j \not \in \{a_i, b_i, c_i\}$. 
In particular, $\Phi( \frac{ \partial r_i }{\partial g_j})$ has at most one entry which is of the form $\alpha t$ for $\alpha \in \Q(\xi_d)$ and all its other entries are elements of $\Q(\xi_d)$. 
It follows that the degree of $$ \det\Big( \left[\pi \left( \Phi\left( \frac{\partial r_i}{\partial g_j}\right) \right)\right]_{i,j=2}^c\Big)$$ is no more than $c(K)-1$ and so the degree of $\widetilde{\Delta}_K^{\chi}(t)$ is no more than $c(K)-3$.

Since polynomials of the form $f(t) \overline{f(t)}$ certainly have
even degrees, either $\widetilde{\Delta}_K^{\chi}(t)$ is not a norm, or
we have
\[
 2 \left \lfloor\frac{c(K)-3}{2} \right  \rfloor =  \deg \pi\left(\widetilde{\Delta}_K^{\chi}(t)\right)  \leq \deg \widetilde{\Delta}_K^{\chi}(t) \leq 2 \left \lfloor\frac{c(K)-3}{2}\right  \rfloor ,
\]
and hence we have equality throughout. 
\end{proof}

Table \ref{comptable} gives the degrees of the irreducible
factors of $\pi(\widetilde{\Delta}_{K_n}^{\chi_{\pm}}(t))$ over
$\Z_s[t^{\pm 1}]$. We refer the reader to Appendix~\ref{section:polycomp} for
exposition of the computational details. 

\def\arraystretch{1.2}
\begin{table}[h!]\label{tab:degrees}
\begin{tabular}{|c|c|c|c|c|}
\hline\vspace{-5mm}&&&&\\\
$n$ & $\pm$ & $s=kn+1$ & $\theta \in \Z_s$ & degree sequence of $\pi\left(\widetilde{\Delta}_{K_n}^{\chi_{\pm}}(t)\right)$\\
\hline
11 & + & 23 & 2 & (2,2,3,3,8)\\
& $-$ &  & 2 & (4,14) \\
\hline
17 & + & 103 & 8 & (2,3,9,16)\\
& $-$ &  & 9 & (2, 28) \\
\hline
23 & + & 47 & 4 & (1, 1,11,29)\\
& $-$ &  & 2 & (1, 1, 2, 12, 12, 14) \\
\hline
\end{tabular}
\vspace{2mm}
\caption{The degree sequences
 of $\pi(\widetilde{\Delta}_{K_n}^{\chi_{\pm}}(t))$.}
\label{comptable}
\end{table}

We are now ready to embark upon proving the main theorems of this paper. 

\begin{proof}[Proof of Theorem~\ref{thm:non-slice}]
Let $n \in \{11, 17, 23\}$ and let $K=K_n$. Let $r \colon X_K \to X_K$
denote the order $n$ symmetry of the knot exterior given in
Figure~\ref{fig:seif} by rotation by $2 \pi/ n$.  As discussed above,
$r$ extends to an order $n$ symmetry of $\Sigma_3(K)$ and induces a
covering transformation invariant, linking form preserving isomorphism
$r_* \colon H_1(\Sigma_3(K);\Z) \to H_1(\Sigma_3(K);\Z)$.  Let $P$ be
a covering transformation invariant metabolizer of
$H_1(\Sigma_3(K);\Z)$. By the discussion preceding
Equation~\eqref{eqn:orbitmet}, we see that either $P=P_+$ or there
exists some $k=0, \dots, n-1$ such that $P= r_*^k(P_-)$.

In the former case, let $\chi_+$ be the character defined in Equation~\eqref{eqn:chidef} and note that $\chi_+$ vanishes on $P=P_+$.  Moreover, the computations in Table~\ref{comptable}, the observation that $2 \lfloor\frac{c(K_n)-3}{2}\rfloor= 2 \lfloor\frac{2n-3}{2}\rfloor= 2(n-2)$, and  Proposition \ref{prop:reduction} together imply that $\tilde{\Delta}_K^{\chi^+}(t)$
does not factor as a norm over $\Q(\xi_n)[t^{\pm1}]$.

In the latter case, let $\chi_- \colon H_1(\Sigma_3(K);\Z) \to \Z_n$ be
the character defined in Equation~\eqref{eqn:chidef} that vanishes on
$P_-$.  Since $r_*^k(P_-)= P$, we have that $\chi:= \chi_- \circ
r_*^k$ vanishes on $P$. Moreover, since $r$ is a diffeomorphism of the
0-surgery, we have that
$\tilde{\Delta}_K^{\chi}(t)=\tilde{\Delta}_K^{\chi_-}(t)$. So again
the computations in Table~\ref{comptable} and Proposition
\ref{prop:reduction} imply that $\tilde{\Delta}_K^{\chi}(t)$ does not
factor as a norm over $\Q(\xi_n)[t^{\pm1}]$.

Therefore, for each invariant metabolizer of $H_1(\Sigma_3(K);\Z)$ we have constructed a character of prime power order vanishing on that metabolizer so that the corresponding twisted Alexander polynomial of $K$ is not a norm. By Theorem~\ref{thm:twistedalexsliceobstruction}, we conclude that $K$ is not slice. \end{proof}

Recall that $J=8_{17}\#8_{17}^r$.   Kirk and Livingston \cite{Kirk-Livingston:1999-2} proved $J$ is not slice by showing that for each invariant metabolizer of $P \leq H_1(\Sigma_3(K);\Z) \cong (\Z_{13})^4$ there exists a character $\chi \colon H_1(\Sigma_3(J); \Z ) \to \Z_{13}$ such that $\chi|_P=0$ and $\widetilde{\Delta}_K^{\chi}(t)\in \Q(\xi_{13})[t^{\pm1}]$ is not a norm. Now we are ready to prove our main theorem.

\begin{proof}[Proof of Theorem~\ref{thm:main2}]
For the duration of this proof we refer to $J$ as $K_{13}$, apologizing to the reader for the inconsistency in notation. 

Suppose that $$K= a_{7} K_7 \# a_{11} K_{11}\#a_{13} K_{13}  \# a_{17} K_{17} \#
a_{23} K_{23}
$$ is slice for $a_7,a_{13},a_{11},a_{17},a_{23} \in \{0,1\}$.
If $a_{11}=a_{13}=a_{17}=a_{23}=0$, then Sartori's work \cite{Sartori} implies that $a_7=0$, since $K_7$ is not slice. So we can assume that there exists $i_0 \in \{11, 13, 17, 23\}$ such that $a_{i_0} \neq 0$.

 Let
$$I:= \{ i \in \{7,11,13,17,23\} \mid a_i \neq 0 \}$$ and $P$ be an invariant metabolizer for $H_1(\Sigma_3(K);\Z)$.

Since \begin{equation*}\label{eqn:connsumdecomp}
H_1(\Sigma_3(K);\Z) \cong  \bigoplus_{i \in I} H_1(\Sigma_3(K_i); \Z) \cong\bigoplus_{i \in I}  \left( \Z_{i}\right)^4,
\end{equation*}
and $7, 11, 13, 17,$ and $23$ are relatively prime,   $P':=P \cap H_1(\Sigma_3(K_{i_0});\Z)$ is an invariant metabolizer for $H_1(\Sigma_3(K_{i_0});\Z)$.

Moreover, if $\chi'\colon H_1(\Sigma_3(K_{i_0}); \Z ) \to \Z_{i_0}$ is a character vanishing on $P'$, then we can construct a character $\chi$ vanishing on $P$ by decomposing $$H_1(\Sigma_3(K);\Z)\cong \bigoplus_{i \in I}  H_1(\Sigma_3(K_i); \Z)$$ and letting  \[ \chi|_{H_1(\Sigma_3(K_i);\Z)}= \begin{cases} \chi' & i=i_0 \\ 0 & i \neq i_0. \end{cases}\] Moreover, for such a character we have $\widetilde{\Delta}_K^{\chi}(t)= \widetilde{\Delta}_{K_{i_0}}^{\chi'}(t).$

It therefore suffices to show that for any invariant metabolizer of $H_1(\Sigma_3(K_{i_0});\Z)$ there exists a character $\chi'$ to $\Z_{i_0}$
vanishing on that metabolizer such that the resulting twisted Alexander polynomial $\widetilde{\Delta}_{K_{i_0}}^{\chi'}(t)$ does not factor as a norm over $\Q(\xi_{i_0})[t^{\pm1}]$. 

This is exactly what we did in the proof of Theorem~\ref{thm:non-slice} for $i_0=11, 17, 23$ and what Kirk and Livingston did in \cite{Kirk-Livingston:1999-2} for the case of $i_0=13$ thereby completing the proof. 
\end{proof}

\begin{appendix}
\section{Computation of twisted Alexander polynomials}\label{section:polycomp}

For the purpose of this argument, it is helpful to have the following naming conventions that are standard in this subfield. Given a knot $K$ in $S^3$ bounding a Seifert surface $F$, we write:
\begin{enumerate}
\item[]$\nu(K)$ to denote an open tubular neighborhood of $K$,
\item[]$\nu(F)$ to denote an open tubular neighborhood of $F$,
\item[]$X_K$ to denote $S^3\smallsetminus\nu(K)$,
\item[]$X^n_K$ to denote the $n$-fold cyclic cover of $X_K$, and
\item[]$X_F$ to denote $S^3\smallsetminus\nu(F)$.
\end{enumerate}

Given a character $\chi\colon H_1(\Sigma_3(K)) \to \Z_n$, we apply~\cite{HeraldKirkLivingston10} to obtain a
representation $$\phi_\chi \colon \pi_1(X_K) \to
GL(3, \Q(\xi_n)[t^{\pm1}])$$
as follows. 
Fix a basepoint $x_0$ in $X_F$ and let $\tilde{x}_0$ denote the lift of
$x_0$ to the $0^{th}$ copy of $S^3 \ssm \nu(F)$ in $X_K^3 \subset
\Sigma_3(K)$.  
Let $\epsilon \colon \pi_1(X_K) \to \Z$ be the canonical
abelianization map, and let $\mu_0$ be a preferred meridian of $K$ based at $x_0$. 
Given a simple closed curve $\gamma$ in $S^3 \ssm K$
based at $x_0$ and with $\lk(K, \gamma)=0$, we can obtain a well-defined
lift $\widetilde{\gamma}$ of $\gamma$ to $\Sigma_3(K)$, giving a map $$l \colon
\ker(\epsilon) \to H_1(\Sigma_3(K); \Z).$$ The map $l$ does not in general coincide with our previous method of converting elements of $H_1(S^3 \ssm \nu(F); \Z)$ to elements of $H_1(\Sigma_3(K); \Z)$, unless $\gamma$ is actually disjoint from $F$. In particular,  $l(\mu_0 g \mu_0^{-1})= t \cdot l(g)$
despite the fact that $\mu_0 g \mu_0^{-1}$ and $g$ certainly represent the same class in $H_1(S^3 \ssm \nu(F); \Z)$.

Our choice of $\mu_0$ allows us to define a map 
\begin{align*}
\phi \colon  \pi_1(X_K) &\to  \Z \ltimes H_1(\Sigma_3(K); \Z) \\
g & \mapsto  (t^{\epsilon(g)}, l(\mu_0^{-\epsilon(g)}g)),
\end{align*}
where  the product structure on $\Z \ltimes H_1(\Sigma_3(K); \Z)$ is given by 
\[
(t^{m_1}, x_1) \cdot (t^{m_2}, x_2)= (t^{m_1+m_2}, t^{-m_2} \cdot x_1+x_2).
\]
We then define $\phi_\chi= f_\chi \circ \phi$, where 
\begin{align}
f_{\chi}  \colon  \Z \ltimes H_1(\Sigma_3(K); \Z) & \to GL(3, \Q(\xi_n)[t^{\pm1}]) \nonumber \\ 
(t^m, x) & \mapsto \left[\begin{array}{ccc}
 0& 0 & t \\ 1 & 0 &0 \\ 0 & 1 & 0 \end{array}
\right]^m
\left[\begin{array}{ccc}
 \xi_n^{\chi(x)} & 0 & 0 \\ 0 & \xi_n^{\chi(t \cdot x)} &0 \\ 0 & 0 & \xi_n^{\chi(t^2 \cdot x)} \end{array}
\right]. \label{eqn:charactertorep}
\end{align}

Our basepoint $x$ for $S^3 \smallsetminus \nu(K)$ lies far below the
diagram, which we think of as lying almost in the plane of the
page. All of our curves are based at $x_0$, though as usual we
sometimes draw meridians to components of the knots as unbased curves,
with the understanding that they are based via the `go straight down to
the basepoint' path.
\begin{figure}[h]
\includegraphics[height=5cm]{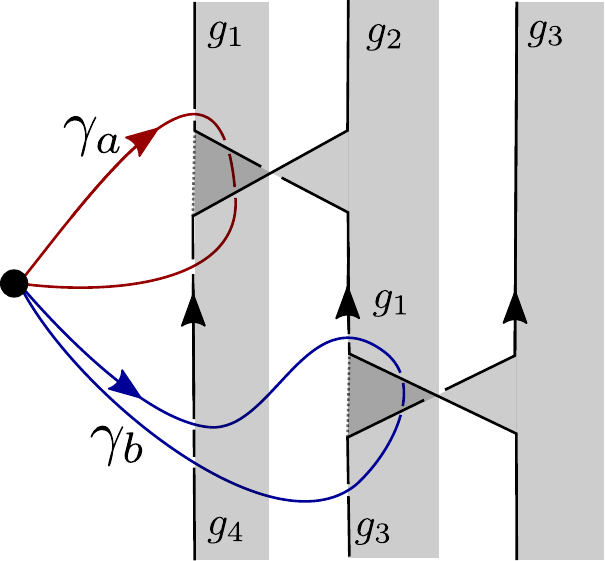}
\caption{Wirtinger generators $g_i$.} \label{fig:wirtingergen}
\end{figure}

 Let $\{g_i\}_{i=1}^{2n}$ be the Wirtinger generators for $\pi_1(X_K,x_0)$, some of which are illustrated in Figure~\ref{fig:wirtingergen}, and $\mu_0$ be the preferred meridian that represents $g_1$. In order to compute our desired twisted Alexander polynomials, we need to know $\phi_\chi(g_i)$ for all $i=1, \dots, 2n$. 
  Since $K$
is the closure of a 3-braid, once we specify the image of the three
top strand generators $g_1, g_2, g_3$ under $\phi_\chi$, the rest of
the computation is simple. In fact, since $g_2=g_1^{-1} g_4 g_1$, it suffices
to determine the image of $g_1, g_3,$ and $g_4$.

By considering Equation~\eqref{eqn:charactertorep}, we see that $\phi_\chi(g_i)$ is determined by the tuple
$$(*)_i:=\left(\chi(l(g_1^{-1} g_i)), \chi(t \cdot l(g_1^{-1} g_i)), \chi( t^2
\cdot l(g_1^{-1} g_i)\right).$$ We now describe $(*)_1, (*)_3$, and $(*)_4$, and use the above discussion to compute $\phi_{\chi}(g_i)$ for each Wirtinger generator $g_i$. We obtain immediately that
\[
(*)_1=
(\chi(l(g_1^{-1} g_1)), \chi(t \cdot l(g_1^{-1} g_1)), \chi( t^2 \cdot
l(g_1^{-1} g_1)))=(0,0,0) \in \mathbb{Z}_n^3.
\] Given a simple closed curve $\gamma$ based at $x_0$ and disjoint from
$F$, recall that we obtain a curve $\widetilde{\gamma}$ in
$\Sigma_3(K)$ by lifting $\gamma$ to our preferred copy of $S^3
\smallsetminus \nu(F)$.  As before, we let $a$ denote the homology
class of the lift of $\widehat{\alpha}_{n-1}$ and $b$ denote the
homology class of the lift of $\widehat{\beta}_{n-1}$ in $H_1(\Sigma_3(K); \Z)$. Let $\gamma_a$ be a simple closed curve that represents $g_1 g_4^{-1}$ and $\gamma_{-a}$ be its reverse, chosen to be disjoint from $F$ as in Figure~\ref{fig:wirtingergen}. 
 Then we have that $-a=\left[\widetilde{\gamma}_a\right] \in H_1(\Sigma_3(K); \Z) $ and 
$$a=\left[ \widetilde{\gamma}_{-a}\right]= l(g_4 g_1^{-1})= l(g_1
(g_1^{-1} g_4) g_1^{-1})= t \cdot l(g_1^{-1} g_4) \in H_1(\Sigma_3(K); \Z).$$
Therefore 
\begin{align*}
(*)_4=
(\chi(l(g_1^{-1} g_4)), \chi(t \cdot l(g_1^{-1} g_4)), \chi( t^2 \cdot l(g_1^{-1} g_4))) &=(\chi(t^{-1}\cdot a),\chi(a),\chi(t \cdot a))\\
&=(-\chi(a)-\chi(t \cdot a),\chi(a),\chi(t \cdot a)) \in \mathbb{Z}_n^3.
\end{align*} Similarly, let $\gamma_b$ be a simple closed curve that represents $g_4 g_3 g_1^{-1} g_4^{-1}$ and is disjoint from $F$, as in Figure~\ref{fig:wirtingergen}.
 So we have that
\begin{align*}
b= \left[ \widetilde{\gamma}_b \right]= l(g_4 g_3 g_1^{-1} g_4^{-1}) = t \cdot l (g_3 g_1^{-1})= t \cdot l(g_1 (g_1^{-1} g_3) g_1^{-1})=t^2 \cdot l(g_1^{-1} g_3) \in H_1(\Sigma_3(K); \Z).
\end{align*}
Hence 
\begin{align*}
(*)_3=
(\chi(l(g_1^{-1} g_3)), \chi(t \cdot l(g_1^{-1} g_3)), \chi( t^2 \cdot l(g_1^{-1} g_3))) &=(\chi(t \cdot b), \chi(t^2 \cdot b), \chi(b))\\
&=(\chi(t \cdot b),- \chi(b)- \chi(t \cdot b), \chi(b)) \in \mathbb{Z}_n^3.
\end{align*}
We can now straightforwardly compute $\phi_\chi(g_i)$ for the rest of the Wirtinger generators $g_i$.

The following well-known result (see \eg \cite{HeraldKirkLivingston10, Kirk-Livingston:1999-1}) reduces computation of twisted Alexander polynomials to Fox calculus and matrix algebra. 

\begin{prop}[{\cite[Section 9]{HeraldKirkLivingston10}}]\label{prop:computation}
Let $\pi_1(X_K) = \langle g_1, \dots, g_c: r_1, \dots , r_c \rangle$
be a Wirtinger presentation .  Assume that $\phi _{\chi} \colon
\pi_1(X_K) \to GL(q, \mathbb{F}[t^{\pm1}])$ is induced by a non-trivial
character $\chi$, and there is a natural extension $\Phi \colon
\Z[\pi_1(X_K)] \to M_q(\mathbb{F}[t^{\pm1}])$ where $M_q(\mathbb{F}[t^{\pm1}])$ is the set of $q$ by $q$ matrices with entries from $\mathbb{F}[t^{\pm1}]$. Then the reduced twisted
Alexander polynomial of $(K, \chi)$ is
\[\pushQED{\qed} \widetilde{\Delta}_K^{\chi}(t)= 
\frac{ \det\Big( \left[ \Phi\left( \frac{\partial r_i}{\partial g_j}\right) \right]_{i,j=2}^c\Big)}{(t-1)\det( \phi_\chi(g_1)).} \qedhere
\]
\end{prop}

The following computations of the  irreducible factors  of the polynomials $\pi(\tilde{\Delta}_K^{\chi^{\pm}}(t)) \in \Z_s[t^{\pm1}]$ were done in Maple worksheets that are available on the third author's website.

\begin{align*}
\begin{array}{cc}
(n, \pm, s, \theta) &\text{Irreducible factors}\\
\hline \hline 
 \bf{(11, -, 23, 2)} \\
  \text{degree }4: &  t^4 + 17 t^3 + 4 t^2 + 17 t + 1 \\
   \text{degree }14: &  t^{14} + 7 t^{13} + 5 t^{12} + 7 t^{11} + 7 t^{10} + 22 t^9+ 22 t^8 + 7 t^7  \\
   &+ 22 t^6 + 22 t^5 + 7 t^4 + 7 t^3 + 5 t^2 + 7 t + 1 \\
   \hline 
   \bf{(11, +, 23, 2)}\\
 \text{degree }2: &  t^2 + 13 t + 1, \,   t^2+3t+11\\
 \text{degree }3: &  t^3 + 14 t^2 + 3,  \, t^3 + 22 t^2 + 22 t + 22 \\
 \text{degree }8: & t^8 + 22 t^7 + 4 t^6 + 14 t^5 + 3 t^4 + 3 t^3 + 16 t^2 + t + 20 
 \end{array}
 \end{align*}

\begin{align*}
\begin{array}{cc}
(n, \pm, s, \theta) &\text{Irreducible factors}\\
\hline \hline 
\bf{(17, +, 103, 8)}\\
    \text{degree }2: & t^2 + 98 t + 5\\
   \text{degree }3: &  t^3 + 12 t^2 + 36 t + 93\\
  \text{degree }9: & t^9 + 33 t^8 + 94 t^7 + 32 t^6 + 61 t^5 + 20 t^4 + 63 t^3  + 48 t^2 + 19 t + 94\\
   \text{degree }16: & t^{16} + 74 t^{15} + 26 t^{14} + 92 t^{13} + 31 t^{12} + 85 t^{11} + 86 t^{10} +  34 t^9 + 35 t^8 \\
 & + 67 t^7 + 99 t^6+64 t^5 + 67 t^4 + 11 t^3 + 95 t^2  + 8 t + 19 \\
\hline
\bf{(17, -, 103, 9)}\\
    \text{degree }2:  & t^2 + 13 t + 1 \\
  \text{degree }28:  & t^{28} + 61 t^{27} + 97 t^{26} + 22 t^{25} + 25 t^{24} + 
27 t^{23} + 73 t^{22} + 47 t^{21} + 79 t^{20} + 31 t^{19} \\
&  + 99 t^{18} + 36 t^{17} + 54 t^{16}  + 40 t^{15} + 40 t^{14} + 40 t^{13} + 54 t^{12} + 36 t^{11}+  99 t^{10} \\
&  + 31 t^9 + 79 t^8 + 47 t^7 + 73 t^6 + 27 t^5 + 25 t^4 + 22 t^3 + 97 t^2 + 61 t + 1 \\
\hline \hline 
\bf{(23, +, 47, 4)}\\
      \text{degree }1:  & t+ 21, \, t + 29 \\
    \text{degree }11: & t^{11} + 37 t^{10} + 43 t^9 + 5 t^8 + t^7 + 42 t^6 + 34 t^5 + 43 t^4 + 5 t^3 + 34 t^2 + 44 t + 9\\
    \text{degree }29: & t^{29} + 25 t^{28} +  9 t^{27} + 19 t^{26} + 38 t^{25} + 46 t^{24} + 27 t^{23} + 40 t^{22}+ 41 t^{21} + 18 t^{20} \\
& + 17 t^{19}  + t^{18} +  34 t^{17}+  6 t^{16} + 21 t^{15} + 25 t^{14} +  18 t^{13}+ 25 t^{12} + 34 t^{11} + 9 t^{10} \\ 
& + 12 t^9  + 41 t^8 + 46 t^7  
+ 10 t^6+ 40 t^5 + 21 t^4 + 10 t^3 + t^2 + 40 t + 13 \\
\hline
\bf{(23, -, 47, 2)}\\
     \text{degree }1: & t + 46,t+46 \\
    \text{degree }2:  &  t^2 + t + 1  \\
    \text{degree } 12: & t^{12} + 3 t^{11} + 27 t^{10} + 19 t^9 + 38 t^8 + 25 t^7 + 25 t^6 + 40 t^5  + 16 t^4+ 25 t^3  \\
    &+ 44 t^2 + 28 t + 23, \, \,  t^{12} + 38 t^{11} + 6 t^{10} +  44 t^9 + 15 t^8 + 14 t^7 + 44 t^6 \\
    &+ 44 t^5 + 18 t^4 + 9 t^3 + 40 t^2 + 41 t + 45 \\
    \text{degree }14: & t^{14} + 2 t^{13} + 2 t^{12} + 43 t^{11} + 
42 t^{10} + 36 t^9 + 30 t^8 + 33 t^7 + 30 t^6 \\
&+ 36 t^5 + 42 t^4 + 43 t^3 + 2 t^2 + 2 t + 1
\end{array}
\end{align*}
\end{appendix}

\bibliography{biblio}
\bibliographystyle{abbrv}

\end{document}